\pdfoutput=1
\documentclass[oneside,11pt]{amsart}
\usepackage{latexsym,amssymb,amsmath,graphicx}
\usepackage{tikz}

\theoremstyle{plain}
\newtheorem*{thm*}{Theorem}\newtheorem{lem}{Lemma}

\newcommand{\conjclass}{\mathcal{C}}
\newcommand{\Z}{\mathbb{Z}}\newcommand{\R}{\mathbb{R}}
\newcommand{\sym}{\textsc{Sym}}\newcommand{\alt}{\textsc{Alt}}
\newcommand{\aut}{\textsc{Aut}}\newcommand{\out}{\textsc{Out}}
\newcommand{\inn}{\textsc{Inn}}

\tikzstyle{names}=[shape=circle,inner sep=1pt,draw,fill=white]
\newcommand{\radA}{1.2cm}
\newcommand{\radB}{.85cm}
\newcommand{\drawDot}[1]{\filldraw[black] #1 circle (2pt);}
\newcommand{\drawWhiteDot}[1]{\drawDot{#1} \filldraw[white] #1 circle (1.4pt);}

\newcommand{\coloredIcos}[3]{      
  \foreach \x in {1,2,...,6} {\coordinate (y\x) at (60*\x:\radA);}
  \foreach \x in {7,8,9} {\coordinate (y\x) at (120*\x-60:\radB);}
  \foreach \x in {10,11,12} {\coordinate (y\x) at (120*\x:\radB);}
  \foreach \x in {1,2,...,6} {\coordinate (z\x) at (60*\x:1.1*\radA);}
  \begin{scope}[thick]
    \filldraw[fill=#1,join=bevel] (y1)--(y2)--(y3)--(y4)--(y5)--(y6)--cycle;
    \filldraw[fill=#2,join=bevel] (y1)--(y2)--(y3)--(y8)--(y7)--cycle;
    \filldraw[fill=#3,join=bevel] (y4)--(y9)--(y6)--(y5)--cycle;
    \draw (y7)--(y8)--(y9)--cycle;
    \draw (y6)--(y7)--(y2)--(y8)--(y4)--(y9)--cycle;
    \draw (y7)--(y1) (y8)--(y3) (y9)--(y5);
  \end{scope}
  \begin{scope}[dashed,thin]
    \draw (y10)--(y11)--(y12)--cycle;
    \draw (y10) edge (y1) edge (y2) edge (y3); 
    \draw (y11) edge (y3) edge (y4) edge (y5); 
    \draw (y12) edge (y5) edge (y6) edge (y1); 
  \end{scope}
}

\newcommand{\yellowIcos}{\coloredIcos{yellow!80!black}{yellow!100!black}{yellow!60!black}}
\newcommand{\blueIcos}{\coloredIcos{blue!70!white}{blue!50!white}{blue!90!white}}

\newcommand{\labelIcos}[6]{
  \def\la{\tiny$#1$}
  \def\lb{\tiny$#2$}
  \def\lc{\tiny$#3$}
  \def\ld{\tiny$#4$}
  \def\le{\tiny$#5$}
  \def\lf{\tiny$#6$}
  \node at (z3) [names]{\la}; \node at (z6) [names]{\la};
  \node at (z1) [names]{\lb}; \node at (z4) [names]{\lb};
  \node at (z2) [names]{\lc}; \node at (z5) [names]{\lc};
  \node at (y7) [names]{\le};
  \node at (y8) [names]{\ld};
  \node at (y9) [names]{\lf};
}

\newcommand{\pairIcos}[7]{
  \draw[fill=red!50,rounded corners=5pt] (-3.2,-1.5) rectangle (3.2,1.5);
  \begin{scope}[xshift=-1.6cm]
    \yellowIcos \labelIcos{#1}{#2}{#3}{#4}{#5}{#6}
  \end{scope}
  \node at (0,1) [fill=white] {$#7$};
  \begin{scope}[xshift=1.6cm]
    \blueIcos \labelIcos{#4}{#5}{#6}{#1}{#2}{#3}
  \end{scope}
}

\newcommand{\shiftedPairIcos}[9]{
    \begin{scope}[xshift=#8,yshift=#9]
      \pairIcos{#1}{#2}{#3}{#4}{#5}{#6}{#7}
    \end{scope}
}

\begin{document}

\title{The exceptional symmetry}
\author{Jon McCammond}
\address{Dept. of Math., University of California, Santa Barbara, CA 93106} 
\email{jon.mccammond@math.ucsb.edu}
\date{\today}


\begin{abstract}
  This note gives an elementary proof that the symmetric groups
  possess only one exceptional symmetry.  I am referring to the fact
  that the outer automorphism group of the symmetric group $\sym_n$ is
  trivial unless $n=6$ and the outer automorphism group of $\sym_6$
  has a unique nontrivial element.
\end{abstract}
\maketitle

When we study symmetric groups, we often invoke their natural faithful
representation as permutations of a set without a second thought, but
to what extent is this representation intrinsic to the structure of
the group and to what extend is it one of several possible choices
available?  Concretely, suppose I am studying the permutations
$\sym_X$ of a set $X = \{1,2,3,4,5,6\}$ and you are studying the
permutations $\sym_A$ of a set $A = \{a,b,c,d,e,f\}$ and suppose
further that we know an explicit isomophism $\phi$ between my group
$\sym_X$ and your group $\sym_A$.  Does this means that there is a way
to identify my set $X$ with your set $A$ which gives rise to the
isomorphism $\phi$?  In other words, must my transpositions correspond
to your transpositions? Must my $3$-cycles correspond to your
$3$-cycles? Or might it be possible that the transposition $(1,2)$ in
my group is sent by the isomorphism $\phi$ to the element
$(a,b)(c,d)(e,f)$ in your group?  The goal of this note is to give an
elementary proof of the fact that yes there is an isomorphism $\phi$
between these two specific groups sending $(1,2)$ to
$(a,b)(c,d)(e,f)$, but that this is essentially the only unexpected
isomorphism among all of the symmetric groups.  In the language of
outer automorphism groups (which we recall below) we give a proof of
the following well-known and remarkable fact.

\begin{thm*}
$\out(\sym_n)$ is trivial for $n \neq 6$ and $\Z/2\Z$ when $n=6$.
\end{thm*}

Recall that the set of all isomorphisms from a group $G$ to itself
form a group $\aut(G)$ under composition called its \emph{automorphism
  group}.  Moreover, in any group we can create an automorphism by
conjugating by a fixed element of $G$.  Such automorphisms are called
\emph{inner automorphisms} and they form a subgroup $\inn(G)$ which is
normal in $\aut(G)$.  These are what one might call the ``expected''
automorphisms.  Note that in the case of the symmetric groups,
conjugating by a permutation corresponds to relabeling the elements of
the set on which it acts.  The quotient group $\out(G) :=
\aut(G)/\inn(G)$ is the group of \emph{outer automorphisms}.  When the
outer automorphism group is trivial it means that there are no
unexpected automorphisms.  When it is non-trivial, each non-trivial
element represents an equivalence class of unexpected automorphisms
which differ from each other by composition with an inner
automorphism.  It is in this sense that the unique non-trivial element
in $\out(\sym_6)$ represents the only unexpected symmetry that the
symmetric groups possess.  Our proof naturally splits into two parts:
restrictions and a construction.  Following the proof we make a few
remarks about the structure of these exceptional automorphisms and we
conclude with pointers to the literature that the interested reader
can pursue.

\section{Restrictions}

The restrictions follow from two easy lemmas about involutions in
symmetric groups.  Recall that the conjugacy classes of elements in
the symmetric group are determined by their \emph{cycle type} and that
the order of a permutation is the least common multiple of the lengths
of the disjoint cycles used to represent it.  In particular, if we let
$\conjclass_j$ denote the elements of $\sym_n$ with cycle structure
$1^i 2^j$ (with, of course, $i + 2j = n$), then these are precisely
the conjugacy classes of order~$2$ elements in $\sym_n$. The set
$\conjclass_1$ is the conjugacy class of transpositions.  Because
automorphisms must preserve order and conjugacy, they end up permuting
the conjugacy classes of each fixed order.  Thus the image of
$\conjclass_1$ under an automorphism of $\sym_n$ must be one of the
classes $\conjclass_j$.  Our first lemma is already an enormous
restriction.

\begin{lem}\label{lem:inner}
  Any automorphism that sends $\conjclass_1$ to $\conjclass_1$ is
  inner.
\end{lem}

\begin{proof}
  When $x$ and $y$ are noncommuting elements in $\conjclass_1$ and
  $z=xyx=yxy$ we call $\{x, y,z\}$ a \emph{dependent} set of
  transpositions.  Consider the maximal independent noncommuting
  subsets of $\conjclass_1$.  In other words, consider the maximal
  subsets $S\subset \conjclass_1$ such that for all distinct elements
  $x,y\in S$: (1) $x$ and $y$ do not commute, and (2) $xyx$ is not in
  $S$.  The key observation is that the set $S_i = \{ (i,j) | j\neq
  i\}$ has these properties for each $i$ and there are no others.  To
  see this note that noncommuting transpositions must share exactly
  one number, say $x=(i,j)$ and $y=(i,k)$, and that the only
  transpositions that do not commute with either $x$ or $y$ are those
  of the form $(i,l)$ with $l\neq j,k$ or the exceptional case
  $(j,k)$---which is ruled out since $(j,k)=xyx$.  Since the subsets
  $S_i$ are the only subsets satisfying these algebraic conditions, an
  automorphism $\phi$ sending $\conjclass_1$ to $\conjclass_1$ must
  permute the subsets $S_i$ among themselves, say
  $\phi(S_i)=S_{\pi(i)}$.  Conjugating $\phi$ by the permutation $\pi$
  produces a conjugate automorphism $\psi$ that fixes each $S_i$
  setwise.  In fact, $\psi$ must fix each $S_i$ pointwise since
  $(i,j)$ is the unique element in the intersection $S_i \cap
  S_j$. Finally, since it fixes a generating set, $\psi$ is the
  identity and $\phi$ is inner.
\end{proof}

One consequence of Lemma~\ref{lem:inner} is that any two automorphisms
$\phi$ and $\psi$ that send $\conjclass_1$ to $\conjclass_j$ differ by
an inner automorphism since $\phi^{-1} \circ \psi$ sends
$\conjclass_1$ to $\conjclass_1$.  The converse also holds: if $\phi$
and $\psi$ differ by an inner automorphism then both send
$\conjclass_1$ to the same conjugacy class $\conjclass_j$ since
conjugation preserves cycle type.  This means that the size of
$\out(\sym_n)$ is completely determined by the list of places that
$\conjclass_1$ can be sent.  The next lemma shows that this list is
very short.

\begin{lem}\label{lem:outer}
  If an automorphism sends $\conjclass_1$ to $\conjclass_j$ with $j>1$
  then $n = 2j =6$.
\end{lem}

\begin{proof}
  The key observation is that for all $x,y \in \conjclass_1$ the order
  of $xy$ is either $1$, $2$, or $3$ so that an automorphism sending
  $\conjclass_1$ to $\conjclass_j$ is only possible when
  $\conjclass_j$ also has this property.  It is easy to find $x,y \in
  \conjclass_j$ whose product has order $j$ (so $j$ is at most $3$),
  and when $n > 2j$ it is also easy to find two elements $x,y \in
  \conjclass_j$ whose product has order $2j+1 >3$ (thus $n$ must equal
  $2j$).  Examples of both types of products are shown in
  Figure~\ref{fig:products}.  Finally, when $j=2$ and $n=4$, there are
  elements in $\conjclass_1$ whose product has order $3$, but the
  three elements in $\conjclass_2$ pairwise commute.  Therefore the
  only possibility is $j=3$ and $n=6$.
\end{proof}

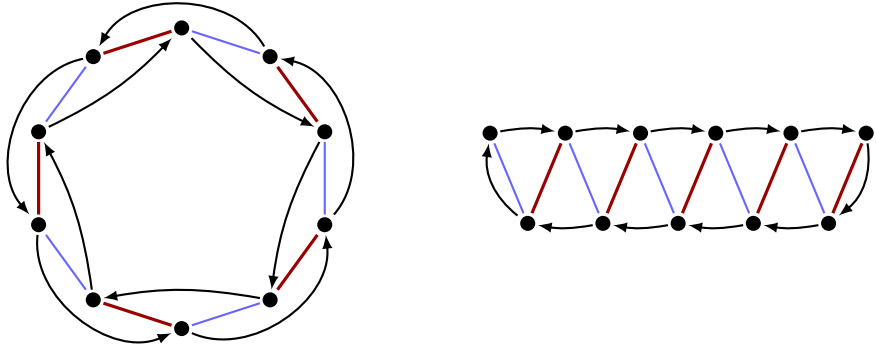
\begin{figure}
  \begin{tikzpicture}
    \newcommand{\rad}{2cm}
    \begin{scope}[xshift=-3cm]
      \foreach \x in {1,2,...,12} {\node (c\x) at (-18+36*\x:\rad) {};}
      \begin{scope}[thick,color=blue!60!white]
      \foreach \x in {2,4,...,10} {\pgfmathtruncatemacro{\xa}{\x+1} \draw (c\x)--(c\xa);}
      \end{scope}
      \begin{scope}[very thick,color=red!60!black]
      \foreach \x in {2,4,...,10} {\pgfmathtruncatemacro{\xz}{\x-1} \draw (c\x)--(c\xz);}
      \end{scope}
      \begin{scope}[->,>=latex,thick]
        \foreach \x in {1,3,...,9} {\pgfmathtruncatemacro{\xb}{\x+2} \draw (c\xb) to [bend right = 10] (c\x);}
        \foreach \x in {2,4,...,10} {\pgfmathtruncatemacro{\xb}{\x+2} \draw (c\x) to [bend right = 60] (c\xb);}
      \end{scope}
      \foreach \x in {1,2,...,10} {\fill (c\x) circle (1mm);}    
    \end{scope}

    \begin{scope}[xshift=3cm]
      \foreach \x in {1,3,...,11} {\node (c\x) at (-2.4+.5*\x,.6) {};}
      \foreach \x in {2,4,...,10} {\node (c\x) at (-2.4+.5*\x,-.6)
        {};}
      \begin{scope}[thick,color=blue!60!white]
      \foreach \x in {1,3,...,9} {\pgfmathtruncatemacro{\xa}{\x+1} \draw (c\x)--(c\xa);}
      \end{scope}
      \begin{scope}[very thick,color=red!60!black]
        \foreach \x in {2,4,...,10} {\pgfmathtruncatemacro{\xa}{\x+1} \draw (c\x)--(c\xa);}
      \end{scope}
      \begin{scope}[->,>=latex,thick]
        \foreach \x in {1,3,...,9} {\pgfmathtruncatemacro{\xb}{\x+2} \draw (c\x) to [bend left=10] (c\xb);}
        \foreach \x in {2,4,...,8} {\pgfmathtruncatemacro{\xb}{\x+2} \draw (c\xb) to [bend left=10] (c\x);}
        \draw (c11) to [bend left] (c10); \draw (c2) to [bend left] (c1); 
      \end{scope}
      \foreach \x in {1,2,...,11} {\fill (c\x) circle (1mm);}    
    \end{scope}
  \end{tikzpicture}
  \caption{The example on the left shows that there are elements in
    $\conjclass_j$ whose product is two $j$-cycles.  The example on
    the right shows that when $n>2j$, there are elements in
    $\conjclass_j$ whose product is a single $(2j+1)$-cycle.  Both
    examples use $j=5$ with $5$ thick dark edges representing one
    element of $\conjclass_5$ and $5$ thin light edges representing
    the other element.\label{fig:products}}
\end{figure}

As a consequence of Lemma~\ref{lem:outer} we know that $\out(\sym_n)$
is trivial for $n\neq 6$ and that $\out(\sym_6)$ has at most two
elements.  The only remaining question is whether or not an
exceptional automorphism of $\sym_6$ sending $\conjclass_1$ to
$\conjclass_3$ actually exists.

\section{A Construction}

An exceptional automorphism of $\sym_6$ that sends $\conjclass_1$ to
$\conjclass_3$ can be constructed using labeled icosahedra.  A regular
icosahedron has twelve vertices that come in $6$ antipodal pairs.
Consider all $6!=720$ ways to label these antipodal pairs by the
numbers $1$ through $6$.  If we identify labelings that differ by a
rigid motion than the number of labelings drops to $12$.  See
Figure~\ref{fig:icosahedra}. This is true whether we include
reflection symmetries or we restrict our attention to rigid motions
that are possible in $\R^3$.  Icosahedra have $120$ symmetries but
because we have restricted our attention to antipodal labelings, the
antipodal map acts trivially on labelings.  Thus only $60$ distinct
labeled icosahedra arise under rigid motions.  Moreover, the antipodal
map, being orientation-reversing, can be composed with any
orientation-reversing isometry to produce an orientation-preserving
one that performs the same modification.

\begin{figure}
  \begin{tikzpicture}[scale=.9]
    \shiftedPairIcos{3}{2}{1}{4}{5}{6}{a}{-3.4cm}{3.4cm}
    \shiftedPairIcos{3}{1}{2}{4}{5}{6}{b}{3.4cm}{3.4cm}
    \shiftedPairIcos{1}{3}{2}{4}{5}{6}{c}{-3.4cm}{0cm}
    \shiftedPairIcos{2}{3}{1}{4}{5}{6}{d}{3.4cm}{0cm}
    \shiftedPairIcos{1}{2}{3}{4}{5}{6}{e}{-3.4cm}{-3.4cm}
    \shiftedPairIcos{2}{1}{3}{4}{5}{6}{f}{3.4cm}{-3.4cm}
  \end{tikzpicture}
  \caption{The twelve antipodal labelings of an icosahedron up to
    isometry grouped into six dual pairs. The six dual pairs are
    labeled $a$ through $f$.\label{fig:icosahedra}}
\end{figure}
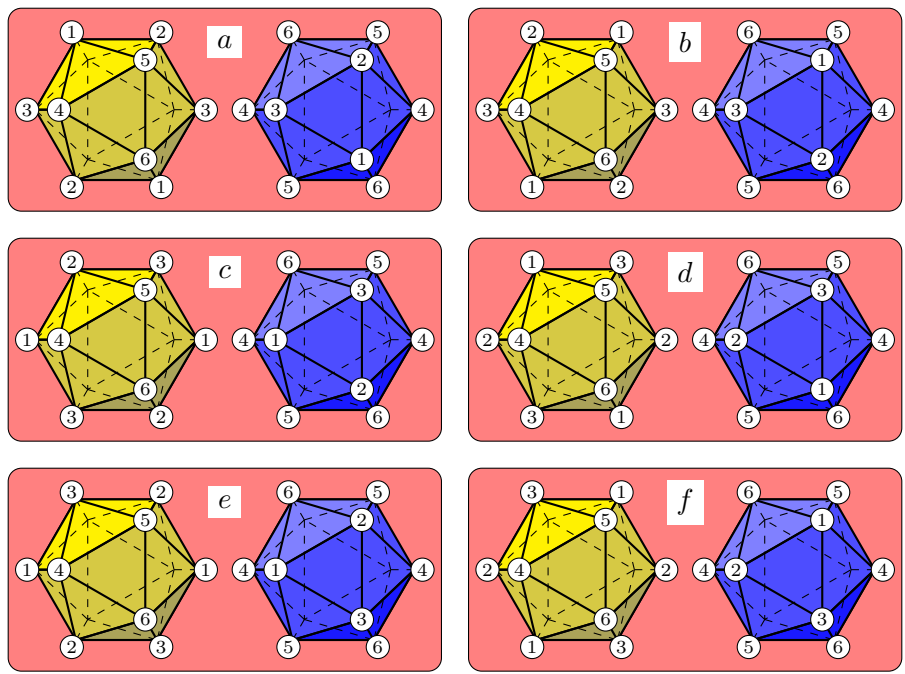

Next, the $12$ antipodal labelings of an icosahedron up to isometry
can be grouped into $6$ pairs.  To see this note that a single labeled
icosahedron contains $20$ labeled triangles but since antipodal
triangles receive the same labels, exactly $10$ out of the possible
$\binom{6}{3} = 20$ labeled triangles actually occur.  It turns out
that the $10$ unused labeled triangles glue together to form one of
the other labeled icosahedra.  An alternative way to see that such a
pairing exists is to consider the complete graph on the $12$ vertices
of an icosahedron with the edges color-coded based on combinatorial
distance in the $1$-skeleton.  The edges representing vertices
distance $1$ apart are the original $1$-skeleton of the icosahedron.
The edges representing vertices distance $3$ apart are a perfect
matching, i.e. $6$ disjoint edges connecting antipodal vertices.  The
remaining edges, representing vertices distance $2$ apart form the
$1$-skeleton of what one might call the \emph{dual icosahedron}.  This
is analogous to the way that the diagonals of a regular pentagon form
another (nonconvex) regular pentagon whose side length has been
multiplied by the golden ratio.  The diagonals of an icosahedron that
connect non-adjacent non-antipodal vertices are the $1$-skeleton of
another (nonconvex) icosahedron.

The $12$ antipodally labeled icosahedra are shown in
Figure~\ref{fig:icosahedra} as $6$ pairs of labeled dual icosahedra
that we identify by the letters $a$ through $f$.  Note that every
possible labeled triangle occurs in one of the two icosahedra in the
pair.  We have colored the icosahedron yellow when it contains a
triangle labeled $456$ and blue when it contains a triangle labeled
$123$.

The symmetry group of the set $X = \{1,2,3,4,5,6\}$ acts on this set
of labeled icosahedra by permuting the vertex labels.  And since this
action of $\sym_X$ respects rigid motions and the dual pairing, every
permutation in $\sym_X$ induces a permutation in $\sym_A$ where $A =
\{a,b,c,d,e,f\}$.  In particular we get a homomorphism $\phi$ from
$\sym_X$ to $\sym_A$.  As an illustration, consider the transposition
$(1,2)$.  It is easy to see from Figure~\ref{fig:icosahedra} that
switching $1$ and $2$ in the labeled icosahedra swaps the dual pair
$a$ and the dual pair $b$, it swaps the dual pair $c$ and the dual
pair $d$ and it swaps the dual pair $e$ and the dual pair $f$.  In
other words, the image of the transposition $(1,2)$ under the map
$\phi$ is the permutation $(a,b)(c,d)(e,f)$ of the labeled dual pairs.

To see that this homomorphism $\phi$ from $\sym_X$ to $\sym_A$ is an
isomorphism, we note that $\sym_6$ has very few normal subgroups.  In
fact, the only non-trivial normal subgroup is $\alt_6$ and the
resulting quotient has size $2$.  Since $\phi$ sends the elements
$(1,2)$, $(1,3)$ and $(2,3)$ to the permutations $(a,b)(c,d)(e,f)$,
$(a,e)(b,c)(d,f)$ and $(a,d)(b,f)(c,e)$ respectively, the image has
size bigger than $2$, the kernel must be trivial and, since both
groups have the same size, the map must be onto and therefore an
isomorphism. In short, this natural construction produces an
isomorphism $\phi$ of $\sym_6$ that sends $\conjclass_1$ in $\sym_X$
to $\conjclass_3$ in $\sym_A$.  Moreover, its inverse, which is also
an isomorphism of $\sym_6$ must send $\conjclass_1$ in $\sym_A$ to a
conjugacy class other than $\conjclass_1$ in $\sym_X$.  By
Lemma~\ref{lem:outer} its image can only be $\conjclass_3$.  In other
words, $\phi$ sends the conjugacy classes $\conjclass_1$ and
$\conjclass_3$ in $\sym_X$ to the conjugacy classes $\conjclass_3$ and
$\conjclass_1$ in $\sym_A$, respectively.

Finally, to turn this isomorphism into an automorphism we simply
identify the letters $a$ through $f$ with the numbers $1$ through $6$
sending $\sym_A$ back to $\sym_X$ in a more traditional fashion.  Note
that the various ways of identifying $A$ and $X$ differ from each
other by an inner automorphism of $\sym_X$ so as we run through the
$6!$ possibilities this procedure actually produces \emph{all} of the
outer automorphisms representing the unique nontrivial element of
$\out(\sym_6)$.

\section{Structure}

The exceptional symmetry of $\sym_6$ has a lot of interesting
structure.  Following Cameron and van Lint (among others) we describe
the various aspects of $\sym_6$ using terminology from graph theory
\cite[Chapter~$6$]{CameronVanLint}.  If we use $X$ (or $A$) to label
the $6$ vertices of a complete graph $K_6$, then the transpositions in
$\conjclass_1$ are its \emph{edges}.  An involution in $\conjclass_3$
corresponds to three disjoint edges which graph theorists would call a
\emph{perfect matching} or a \emph{$1$-factor} or simply a
\emph{factor}.  The $6$ sets $S_i$ of $5$ edges with a common endpoint
that we used in the proof of Lemma~\ref{lem:inner} as an algebraic
replacement for points are called \emph{stars} and the $6$ ways to
partition the $15$ edges of $K_6$ into $5$ disjoint factors are called
\emph{factorizations}.  An exceptional automorphism of $\sym_6$ swaps
the $15$ edges and the $15$ factors and it swaps the $6$ stars and the
$6$ factorizations.  Composing this automorphism with itself produces
an inner automorphism, but the result is not necessarily the identity
map.

\begin{figure}
  \begin{tabular}{ccc}
    \includegraphics[scale=.5]{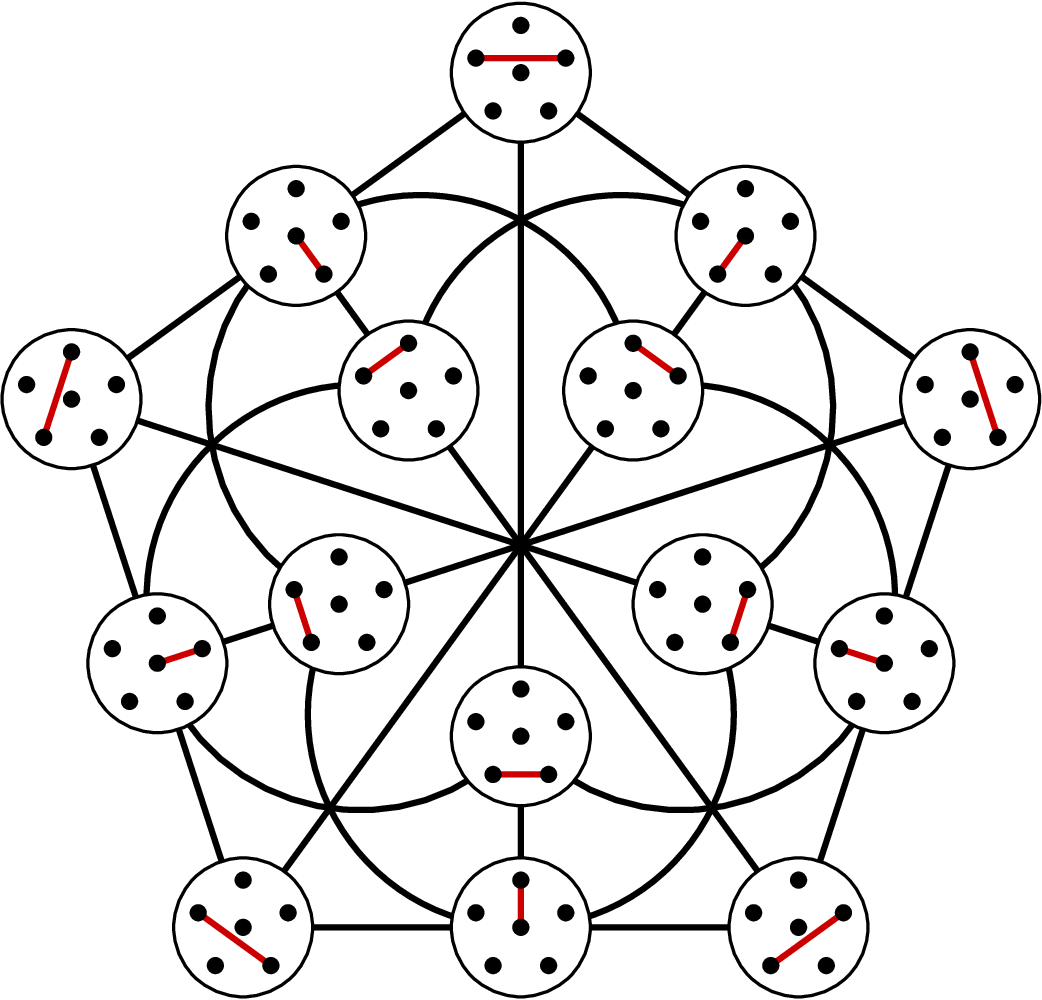} 
    & \hspace*{1em} &
    \includegraphics[scale=.5]{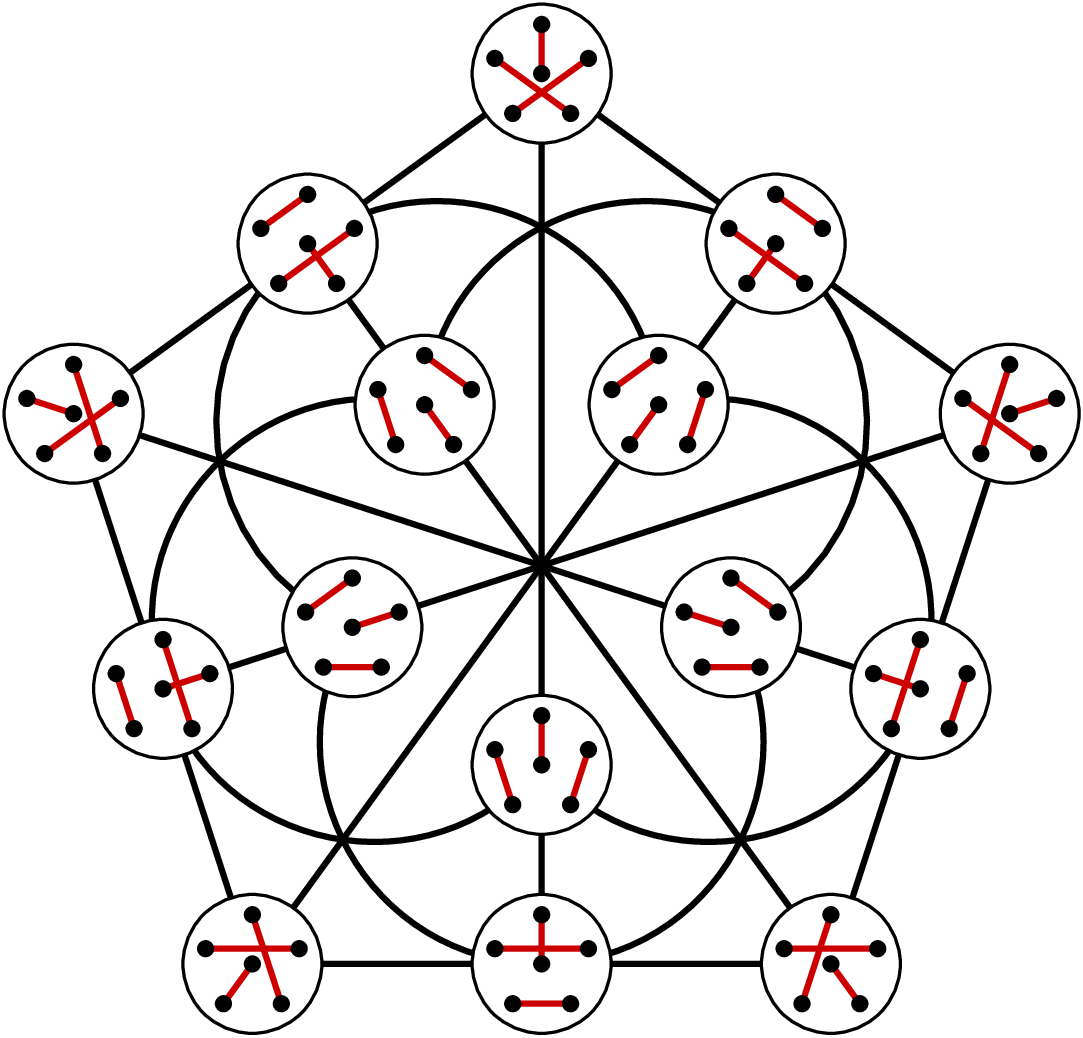}
  \end{tabular}
  \caption{The edge and factor versions of the doily.  The labels of
    the ``points'' are inscribed in the small discs and the line
    segments and circular arcs represent the
    ``lines''.\label{fig:doily}}
\end{figure}

There are, however, some exceptional automorphisms whose square is the
identity ($36$ of them to be precise) and we demonstrate their
existence with the help of an auxillary graph.  The edges and factors
can be used to define an example of a partial geometry known as a
\emph{generalized quadrangle} and this particular example is called
$GQ(2,2)$.  It uses the edges as points and the factors as lines (or
the other way around).  Both versions are shown in
Figure~\ref{fig:doily} in a representation that Stan Payne dubbed
``the doily''.  The incidence graph of this geometry is a bipartite
graph with $15$ white vertices representing edges and $15$ black
vertices representing factors known as \emph{Tutte's $8$-cage}.  A
white vertex is connected to a black vertex if and only if the
corresponding edge belongs to the corresponding factor.  See
Figure~\ref{fig:tutte}.  The automorphism group of the Tutte graph is
precisely the group $\aut(\sym_6)$ of size $1440$.  In particular, the
outer automorphisms of $\sym_6$ correspond to symmetries of this graph
that switch the white and black vertices.  One such symmetry is the
reflection across the vertical axis of Figure~\ref{fig:tutte} and this
clearly corresponds to an exceptional automorphism of $\sym_6$ whose
square is the identity.  An exceptional automorphism that is equal to
its own inverse reminds one of a polarity in projective geometry that
establishes a bijection between points and lines and there are ways to
make this resemblance precise.

\begin{figure}
  \begin{tikzpicture}
    \def\n{60}; \def\na{{(360/\n)}}; \def\ra{3}
    \foreach \x in {1,3,...,59} {\draw (90+6*\x:3) -- (102+6*\x:3);}
    \foreach \x in {0,1,2,3,4} {
      \foreach \y in {7,13,21} {
        \draw (90+72*\x-6*\y:3) -- (90+72*\x+6*\y:3);
      }
    }
    \foreach \x in {1,5,...,57} {
      \drawWhiteDot{(90+6*\x:3)}
      \drawDot{(78+6*\x:3)}
    }
  \end{tikzpicture}
  \caption{The incidence graph of the doily is known as Tutte's
    $8$-cage. A reflection across the vertical axis of symmetry
    illustrates the duality between edges and
    factors.\label{fig:tutte}}
\end{figure}
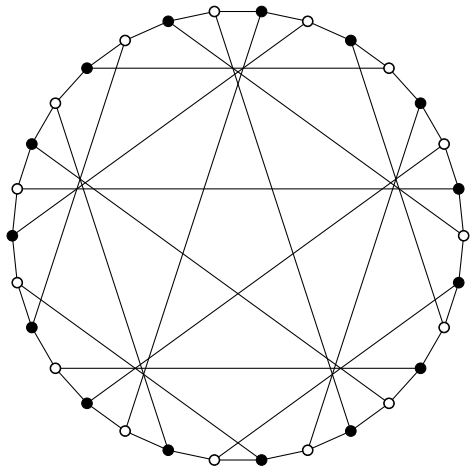

\section{Connections}

There is much more that can be said about the exceptional symmetry of
$\sym_6$, but in this final section I merely make a few remarks about
the connections this symmetry has with other exceptional objects
coupled with a few pointers to some standard references in the
literature.  For those wishing to read more about the exception
symmetry of $\sym_6$ at an accessible level, I highly recommend
Cameron and van Lint's book ``Designs, Graphs, Codes and their Links''
\cite{CameronVanLint}, especially Chapter~$6$, which is called ``A
property of the number six''.  In that chapter, the authors construct
the exceptional symmetry of $\sym_6$ and use these automorphisms to
construct the unique projective plane of order $4$, the $50$ vertex
graph known as the Hoffman-Singleton graph its with many remarkable
properties, and the $S(5,6,12)$ Steiner system whose automorphism
group is the Mathieu group $M_{12}$, one of the smallest and simplest
of the sporadic finite simple groups.  Another good source for some of
the same material is the book on ``Algebraic Graph Theory'' by Godsil
and Royle \cite{GodsilRoyle}.  For explicit details of the
automorphisms themselves and for many references to the early
literature (going back to Sylvester in 1844), I recommend two articles
by H.S.M. Coxeter that are collected as Chapters $6$ and $7$ in his
book ``The beauty of geometry: twelve essays'' \cite{Coxeter}.  Online
there is a post written by John Baez in 1992 called ``Some thoughts on
the number $6$'' \cite{Baez} which is similar in spirit to the
material presented here and the labeled icosahedra construction is one
of several constructions given in the recent article by Howard,
Millson, Snowden and Vakil \cite{HMSV}.  Finally, for the truly
adventurous, I recommend Conway and Sloane's book on ``Sphere packings
lattices and groups'' (particular Chapter 10 called ``Three lectures
on exceptional groups'') \cite{ConwaySloane} and the entry for
$\alt_6$ in the ``Atlas of finite groups'' \cite{ATLAS}.  Both contain
a wealth of material that place the outer automorphism of $\sym_6$ in
a much, much larger context.

\newcommand{\etalchar}[1]{$^{#1}$}
\providecommand{\bysame}{\leavevmode\hbox to3em{\hrulefill}\thinspace}
\providecommand{\MR}{\relax\ifhmode\unskip\space\fi MR }
\providecommand{\MRhref}[2]{%
  \href{http://www.ams.org/mathscinet-getitem?mr=#1}{#2}
}
\providecommand{\href}[2]{#2}


\begin{thebibliography}{HMSV08}

\bibitem[Bae]{Baez}
John Baez, \emph{Some thoughts on the number 6}, Available online at {\tt
  http://math.ucr.edu/home/baez/six.html}.

\bibitem[CCN{\etalchar{+}}85]{ATLAS}
J.~H. Conway, R.~T. Curtis, S.~P. Norton, R.~A. Parker, and R.~A. Wilson,
  \emph{Atlas of finite groups}, Oxford University Press, Eynsham, 1985,
  Maximal subgroups and ordinary characters for simple groups, With
  computational assistance from J. G. Thackray. \MR{827219 (88g:20025)}

\bibitem[Cox99]{Coxeter}
H.~S.~M. Coxeter, \emph{The beauty of geometry}, Dover Publications, Inc.,
  Mineola, NY, 1999, Twelve essays, Reprint of the 1968 original [{{\i}t Twelve
  geometric essays}, Southern Illinois Univ. Press, Carbondale, IL, 1968;
  MR0310745 (46 \#9843)]. \MR{1717154 (2000e:51001)}

\bibitem[CS99]{ConwaySloane}
J.~H. Conway and N.~J.~A. Sloane, \emph{Sphere packings, lattices and groups},
  third ed., Grundlehren der Mathematischen Wissenschaften [Fundamental
  Principles of Mathematical Sciences], vol. 290, Springer-Verlag, New York,
  1999, With additional contributions by E. Bannai, R. E. Borcherds, J. Leech,
  S. P. Norton, A. M. Odlyzko, R. A. Parker, L. Queen and B. B. Venkov.
  \MR{1662447 (2000b:11077)}

\bibitem[CvL91]{CameronVanLint}
P.~J. Cameron and J.~H. van Lint, \emph{Designs, graphs, codes and their
  links}, London Mathematical Society Student Texts, vol.~22, Cambridge
  University Press, Cambridge, 1991. \MR{1148891 (93c:05001)}

\bibitem[GR01]{GodsilRoyle}
Chris Godsil and Gordon Royle, \emph{Algebraic graph theory}, Graduate Texts in
  Mathematics, vol. 207, Springer-Verlag, New York, 2001. \MR{MR1829620
  (2002f:05002)}

\bibitem[HMSV08]{HMSV}
Ben Howard, John Millson, Andrew Snowden, and Ravi Vakil, \emph{A description
  of the outer automorphism of {$S_6$}, and the invariants of six points in
  projective space}, J. Combin. Theory Ser. A \textbf{115} (2008), no.~7,
  1296--1303. \MR{2450346 (2009h:14081)}

\end{thebibliography}
\end{document}